\documentclass{amsart}

\input xy
\xyoption{all}

\usepackage{geometry}
\usepackage{amsmath}
\usepackage{amssymb}
\usepackage{cite}
\usepackage{amsthm}  % see geometry.pdf on how to lay out the page. There's lots.
\usepackage{tikz}
\usetikzlibrary{er,positioning}

\newtheorem{same}{This should never appear}[section]
\newtheorem{defin}[same]{Definition}

\newtheorem{remark}[same]{Remark}
\newtheorem{theorem}[same]{Theorem}

\newtheorem{lemma}[same]{Lemma}
\newtheorem{fact}[same]{Fact}
\newtheorem{question}[same]{Question}
\newtheorem{cor}[same]{Corollary}
\newtheorem{prop}[same]{Proposition}

\newtheorem{nota}[same]{Notation}

\newtheorem{defin*}{Definition}
\newtheorem*{theorem*}{Theorem}
\newbox\noforkbox \newdimen\forklinewidth
\setbox0\hbox{$\textstyle\smile$}
\setbox1\hbox to \wd0{\hfil\vrule width \forklinewidth depth-2pt
 height 10pt \hfil}
\setbox\noforkbox\hbox{\lower 2pt\box1\lower 2pt\box0\relax}
\def\unionstick{\mathop{\copy\noforkbox}\limits}

\def\nonfork_#1{\unionstick_{\textstyle #1}}

\setbox0\hbox{$\textstyle\smile$}
\setbox1\hbox to \wd0{\hfil{\sl /\/}\hfil}
\setbox2\hbox to \wd0{\hfil\vrule height 10pt depth -2pt width
              \forklinewidth\hfil}
\newbox\doesforkbox
\setbox\doesforkbox\hbox{\lower 2pt\box1 \lower 2pt\box2\lower2pt\box0\relax}
\def\nunionstick{\mathop{\copy\doesforkbox}\limits}

\def\fork_#1{\nunionstick_{\textstyle #1}}

\makeatletter
\newcommand{\skipitems}[1]{%
  \addtocounter{\@enumctr}{#1}%
}

\newcommand{\ba}{\bold{a}}
\newcommand{\bb}{\bold{b}}

\newcommand{\rest}{\upharpoonright}

\newcommand{\K}{\mathbf{K}}
\newcommand{\LS}{\operatorname{LS}}

\newcommand{\leap}[1]{\le_{#1}}

\newcommand{\lea}{\leap{\K}}

\newcommand{\gtp}{$ga-tp$}
\newcommand{\gS}{$gS$}

\title{Algebraic description of limit models in classes of abelian groups }
\date{\today.} % delete this line to display the current date

\author{Marcos Mazari-Armida}
\email{mmazaria@andrew.cmu.edu}
\urladdr{http://www.math.cmu.edu/~mmazaria/ }
\address{Department of Mathematical Sciences \\ Carnegie Mellon University \\ Pittsburgh, Pennsylvania, USA}

\begin{document}

\maketitle

{\let\thefootnote\relax\footnote{{AMS 2010 Subject Classification: Primary 03C48. Secondary: 03C45, 20K20.
}}}  

%%%%%%%%%%%%%%%%%%%%%%%%%%%%%%%%%%%%%%%%%%%%%%%
\begin{abstract}

We study limit models in the class of abelian groups with the subgroup relation and in the class of torsion-free abelian groups with the pure subgroup relation. We show:
\begin{theorem}\
\begin{enumerate}
\item If $G$ is a limit model of cardinality $\lambda$ in the class of abelian groups with the subgroup relation, then $G \cong (\oplus_{\lambda}\mathbb{Q}) \oplus \oplus_{p \text{ prime}} (\oplus_{\lambda} \mathbb{Z}(p^\infty))$.
\item  If $G$ is a limit model of cardinality $\lambda$ in the class of torsion-free abelian groups with the pure subgroup relation, then:
\begin{itemize}
\item If the length of the chain has uncountable cofinality, then \[ G \cong (\oplus_{\lambda} \mathbb{Q}) \oplus \Pi_{p \text{ prime}} \overline{(\oplus_{\lambda} \mathbb{Z}_{(p)})}.\]
\item If the length of the chain has countable cofinality, then $G$ is not algebraically compact.
\end{itemize}
\end{enumerate}
\end{theorem}
We also study the class of finitely Butler groups with the pure subgroup relation, we show that it is an AEC, Galois-stable and $(<\aleph_0)$-tame and short.
\end{abstract}
%%%%%%%%%%%%%%%%%%%%%%%%%%%%%%%%%%%%%%%%%%%%%%%

\tableofcontents

%%%%%%%%%%%%%%%%%%%%%%%%%%%%%%%%%%%%%%%%%%%%%%%
\section{Introduction}
%%%%%%%%%%%%%%%%%%%%%%%%%%%%%%%%%%%%%%%%%%%%%%%

Abstract elementary classes (AECs for short) were introduced in the late seventies  by Shelah \cite{sh88} to capture the semantic structure of non-first-order theories, Shelah was interested in capturing logics like $\mathbb{L}_{\lambda^+, \omega}(\textbf{Q})$. The setting is general enough to encompass many examples, but it still allows a development of a rich theory as witnessed by Shelah's two volume book on the subject \cite{shelahaecbook} and many dozens of publications by several researchers. As a first approximation, an AEC is a class of structures with morphisms that is closed under colimits and such that every set is contained in a small model in the class.

\begin{defin}\label{aec-def}
  An \emph{abstract elementary class} is a pair $\K = (K, \lea)$, where:

  \begin{enumerate}
    \item $K$ is a class of $\tau$-structures, for some fixed language $\tau = \tau (\K)$. 
    \item $\lea$ is a partial ordering on $K$. 
    \item $(K, \lea)$ respects isomorphisms: If $M \lea N$ are in $K$ and $f: N \cong N'$, then $f[M] \lea N'$. In particular (taking $M = N$), $K$ is closed under isomorphisms.
    \item If $M \lea N$, then $M \subseteq N$. 
    \item Coherence: If $M_0, M_1, M_2 \in K$ satisfy $M_0 \lea M_2$, $M_1 \lea M_2$, and $M_0 \subseteq M_1$, then $M_0 \lea M_1$.
    \item Tarski-Vaught axioms: Suppose $\delta$ is a limit ordinal and $\{ M_i \in K : i < \delta \}$ is an increasing chain. Then:

        \begin{enumerate}

            \item $M_\delta := \bigcup_{i < \delta} M_i \in K$ and $M_i \lea M_\delta$ for every $i < \delta$.
            \item\label{smoothness-axiom}Smoothness: If there is some $N \in K$ so that for all $i < \delta$ we have $M_i \lea N$, then we also have $M_\delta \lea N$.

        \end{enumerate}

    \item L\"{o}wenheim-Skolem-Tarski axiom: There exists a cardinal $\lambda \ge |\tau(\K)| + \aleph_0$ such that for any $M \in K$ and $A \subseteq |M|$, there is some $M_0 \lea M$ such that $A \subseteq |M_0|$ and $\|M_0\| \le |A| + \lambda$. We write $\LS (\K)$ for the minimal such cardinal.
  \end{enumerate}
\end{defin}

The main objective in the study of AECs is to develop a classification theory like the one of first-order model theory. The notions of non-forking, superstability and stability have been extended to this more general setting. The main test question is Shelah's eventual categoricity conjecture which asserts that if an AEC is categorical in \emph{some} large cardinal then it is categorical in \emph{all} large cardinals. Many partial results have been obtained in this direction as witnessed by for example \cite{sh87a}, \cite{sh87b}, \cite{sh394}, \cite{shelahaecbook}, \cite{tamenesstwo}, \cite{tamenessthree}, \cite{bont}, \cite{vaseye}, \cite{vaseyf},\cite{vaseyu},   \cite{vasey18} and \cite{shvas}.\footnote{For a more detailed introduction to the theory of AECs we suggest the reader to look at \cite{grossberg2002}, \cite{baldwinbook09} or \cite{survey} (this only covers tame AECs, but the AECs that we will study in this paper are all tame).}

The notion of limit model was introduced in \cite{kosh} as a substitute for saturation in the non-elementary setting (see Definition \ref{limit}). If  $\lambda > \LS(\K)$ is a regular cardinal and $\K$ is an AEC with joint embedding, amalgamation and no maximal models, then: $M$ is $\lambda$-Galois-saturated if and only if $M$ is a $(\lambda, \lambda)$-limit model (\cite[2.8]{grva}).

Limit models have proven to be an important concept in tackling Shelah's eventual categoricity conjecture as witnessed by for example \cite{shvi}, \cite{tamenessone} and \cite{vasey18}. The key question has been the uniqueness of limit models of the same cardinality but with chains of different lengths. This has been studied thoroughly \cite{shvi}, \cite{van06}, \cite{grvavi}, \cite{extendingframes}, \cite{vand}, \cite{bovan}, \cite{viza} and \cite{vasey18}. In this same line, \cite{grva} and \cite{vas16a} showed that if a class has a monster model and is tame then uniqueness of limit models on a tail of cardinals is equivalent to being Galois-superstable\footnote{We say that $\K$ is
\emph{Galois-superstable} if there is $\mu < \beth_{(2^{\LS(\K)})^+}$
such that $\K$ is $\lambda$-Galois-stable for every $\lambda\geq \mu$. Under
the assumption of joint embedding, amalgamation, no maximal models
and $\LS(\K)$-tameness (which hold for all the classes studied in this paper, except perhaps the one introduced in the last section) by \cite{grva} and \cite{vaseyt} the definition
of the previous line is equivalent to any other definition of
Galois-superstability given in the context of AECs.}.

Despite the importance of limit models in the understanding of AECs, explicit examples have never been studied. This paper ends this by studying  examples of limit models in some classes of abelian groups. The need to analyze examples is also motivated by the regular inquiry of the model theory community when presenting results on AECs. In particular, the analysis of limit models in the class of torsion-free abelian groups provides a missing example needed for\cite{bovan}.

 In this article, we study limit models in the class of abelian groups with the subgroup relation and in the class of torsion-free abelian groups with the pure subgroup relation\footnote{Recall that $H$ is a pure subgroup of $G$ if for every $n \in \mathbb{N}$ it holds that $nG \cap H  = nH$.}. Observe that both classes are first-order axiomatizable, but since we are studying them with a strong substructure relation that is different from elementary substructure, their study is outside of the framework of first-order model theory. This freedom in choosing the strong substructure relation is a key feature of our examples and in the context of AECs has only been exploited in \cite{grp} and \cite{baldwine}.

The case of limit models in the class of abelian groups is simple.

\textbf{Theorem \ref{abelian}.}
\textit{Let $\alpha < \lambda^+$ a limit ordinal. If $G$ is a $(\lambda, \alpha)$-limit model in the class of abelian groups with the subgroup relation, then we have that:
\[ G \cong (\oplus_{\lambda}\mathbb{Q}) \oplus \oplus_{p \text{ prime}} (\oplus_{\lambda} \mathbb{Z}(p^\infty)).\]}

The case of torsion-free abelian groups (with the pure subgroup relation) is more interesting and the examination of limit models is divided into two cases. In the first one, we study limit models with chains of uncountable cofinality and by showing that they are algebraically compact we are able to give a full structure theorem. In the second one, we study limit models with chains of countable cofinality and we show that they are not algebraically compact. More precisely we obtain the following.

\textbf{Theorem \ref{tf}.}
\textit{Let $\alpha < \lambda^+$ a limit ordinal. If $G$ is a $(\lambda, \alpha)$-limit model in the class of torsion-free abelian groups with the pure subgroup relation, then we have that:
\begin{enumerate}
\item If the cofinality of $\alpha$ is uncountable, then  \[ G \cong  (\oplus_{\lambda} \mathbb{Q}) \oplus \Pi_{p \text{ prime}} \overline{(\oplus_{\lambda} \mathbb{Z}_{(p)})}.\]
\item If the cofinality of $\alpha$ is countable, then $G$ is not algebraically compact.
\end{enumerate}}
In particular, the class does not have uniqueness of limit models for any infinite cardinal. 
\vspace{2mm}
\\
The paper is organized as follows. Section 2 presents necessary background. Section 3 characterizes limit models in the class of abelian groups with the subgroup relation. Section 4 studies the class of torsion-free abelian groups with the pure subgroup relation. We show that limit models of uncountable cofinality are algebraically compact (and characterize them) while those of countable cofinality are not.  Section 5 studies basic properties of the class of finitely Butler groups.

This paper was written while the author was working on a Ph.D. under the direction of Rami Grossberg at Carnegie Mellon University and I would like to thank Professor Grossberg for his guidance and assistance in my research in general and in this work in particular. I would also like to thank John T. Baldwin, Hanif Cheung, Sebastien Vasey and an anonymous referee for valuable comments that significantly improved the paper.

\section{Preliminaries}
We present the basic concepts of abstract elementary classes that are used in this paper. These are further studied in \cite[\S 4 - 8]{baldwinbook09} and  \cite[\S 2, \S 4.4]{ramibook}. Regarding the background on abelian groups, we assume that the reader has some familiarity with it and introduce the necessary concepts throughout the text.\footnote{An excellent encyclopedic resource is \cite{fuchs}. We recommend the reader to keep a copy of \cite{fuchs} nearby since we will cite frequently from it, specially in the last section.}

\subsection{Basic notions} 
Before we introduce some concepts let us fix some notation.

\begin{nota}\
\begin{itemize}
\item If $M \in K$, $|M|$ is the underlying set of $M$.
\item If $\lambda$ is a cardinal, $\K_{\lambda} =\{ M \in K : \| M \| =\lambda \}$.
\item Let $M, N \in K$. If we write ``$f: M \to N$" we assume that $f$ is a $\K$-embedding, i.e., $f: M \cong f[M]$ and $f[M] \lea N$. Observe that in particular $\K$-embeddings are always monomorphisms.
\end{itemize}
\end{nota}
All the examples that we consider in this paper have the additional property of admitting intersections. This class of AECs was introduced in \cite{bash} and further studied in \cite[\S 2]{vaseyu}.

\begin{defin} An AEC \emph{admits intersections} if for every $N \in K$ and $A \subseteq |N|$ there is $M_0 \lea N$ such that $|M_0|= \bigcap\{M \lea N : A \subseteq |M|\}$. For $N \in K$ and $A \subseteq |N|$, we denote by  $cl^{N}_{\K}(A)=\bigcap\{M \lea N : A \subseteq |M|\}$, if it is clear from the context we will drop the $\K$.
\end{defin}

Since an AEC is a semantic object, the notion of syntactic type (first-order type) does not interact well with the strong substructure relation of the AEC. Even when the AEC is axiomatizable in some extension of first-order logic, syntactic types do not behave well since equality of types does not imply the existence of $\K$-embeddings between the models mentioned in the types. For this reason Shelah introduced a notion of semantic type called Galois-type. We use the terminology of \cite[2.5]{mv}.

\begin{defin}\label{gtp-def}
  Let $\K$ be an AEC.
  
  \begin{enumerate}
    \item Let $\K^3$ be the set of triples of the form $(\bb, A, N)$, where $N \in K$, $A \subseteq |N|$, and $\bb$ is a sequence of elements from $N$. 
    \item For $(\bb_1, A_1, N_1), (\bb_2, A_2, N_2) \in \K^3$, we say $(\bb_1, A_1, N_1)E_{\text{at}} (\bb_2, A_2, N_2)$ if $A := A_1 = A_2$, $\ell(\bb_1)=\ell(\bb_2)$ and there exists $f_\ell : N_\ell \xrightarrow[A]{} N$ such that $f_1 (\bb_1) = f_2 (\bb_2)$.
    \item Note that $E_{\text{at}}$ is a symmetric and reflexive relation on $\K^3$. We let $E$ be the transitive closure of $E_{\text{at}}$.
    \item For $(\bb, A, N) \in \K^3$, let $\gtp_{\K} (\bb / A; N) := [(\bb, A, N)]_E$. We call such an equivalence class a \emph{Galois-type}. Usually, $\K$ will be clear from context and we will omit it.
\item For $\gtp_{\K} (\bb / A; N)$ and $C \subseteq A$, $\gtp_{\K} (\bb / A; N)\upharpoonright_{C}:= [(\bb, C, N)]_E$.
  \end{enumerate}
\end{defin}

In classes that admit intersections types are easier to describe as it was shown in \cite[2.18]{vaseyu}. 

\begin{fact}\label{tp-int}
Let $\K$ be an AEC that admits intersections. $\gtp(\ba_1/ A ; N_1)=\gtp(\ba_2/A ; N_2)$ if and only if there is $f: cl^{N_1}(\ba_1 \cup A) \cong_{A} cl^{N_2}(\ba_2 \cup A)$ such that $f(\ba_1)=\ba_2$.
\end{fact}

The notion of Galois-stability generalizes that of a stable first-order theory. Since it will play an important role, as witness by Fact \ref{existence}, we recall it.

\begin{defin}\
\begin{itemize}
\item An AEC is \emph{$\lambda$-Galois-stable} if for any $M \in \K_\lambda$ it holds that $| \gS(M) | \leq \lambda$, where $\gS(M) = \{ \gtp(a/M; N) : M \lea N \text{ and } a \in N\}$. Observe that $\gS(M)$ denotes the 1-ary Galois-types over $M$.
\item An AEC is \emph{Galois-stable} if there is a $\lambda \geq \LS(\K)$ such that $\K$ is $\lambda$-Galois-stable.
%\item An AEC is Galois superstable if there is a $\chi <  \beth_{(2^{\LS(\K)})^+}$ such that $\K$ is $\lambda$-Galois-stable for %every $\lambda \geq \chi$.
\end{itemize}
\end{defin}

Tameness (for saturated models) appears implicitly in the work of Shelah \cite{sh394}, but it was not until  
Grossberg and VanDieren isolated it in \cite{tamenessone} that it became a central notion in the study of  AECs. Tameness was first used to prove a stability spectrum theorem in \cite{tamenessone} and to prove an upward categoricity transfer theorem in \cite{tamenesstwo}. For further details on tameness the reader can consult the survey by Boney and Vasey
\cite{survey}.

\begin{defin}
 $\K$ is \emph{$(< \kappa)$-tame} if for any $M \in K$ and $p \neq q \in \gS(M)$,  there is $A \subseteq M$ such that $|A |< \kappa$ and $p\upharpoonright_{A} \neq q\upharpoonright_{A}$.
\end{defin}

Later, Boney isolated an analogous notion to tameness which he called type shortness in \cite{bont}.

\begin{defin}
$\K$ is \emph{$(<\kappa)$-short} if for any $M, N \in K$, $\bar{a} \in M^{\alpha}$, $\bar{b} \in N^{\alpha}$  and $\gtp(\bar{a}/\emptyset, M )\neq \gtp(\bar{b}/\emptyset, N)$, there is $I \subseteq \alpha$ such that $|I| < \kappa$ and $\gtp(\bar{a}\upharpoonright_I /\emptyset ;  M )\neq \gtp(\bar{b}\upharpoonright_I/\emptyset ; N)$.
\end{defin}

\subsection{Limit models} Before introducing the concept of limit model we recall the concept of universal model.

\begin{defin}
$M$ is \emph{universal over} $N$ if and only if $N \lea M$, $\| M \|=\| N\| =\lambda$ and for any $N^* \in \K_\lambda$ such that $N \lea N^*$, there is $f: N^* \xrightarrow[N]{} M$. 
\end{defin}

Recall that an increasing chain  $\{ M_i : i < \alpha\}\subseteq \K$ (for $\alpha$ an ordinal) is a \emph{continuous chain} if $M_i =\bigcup_{j < i} M_j$  for every $i < \alpha$ limit ordinal.
With this we are ready to introduce the main concept of this paper, it was originally introduced in \cite{kosh}.

\begin{defin}\label{limit}
Let $\alpha < \lambda^+$ a limit ordinal.  $M$ is a \emph{$(\lambda, \alpha)$-limit model over} $N$ if and only if there is $\{ M_i : i < \alpha\}\subseteq \K_\lambda$ an increasing continuous chain such that $M_0 :=N$, $M_{i+1}$ is universal over $M_i$ for each $i < \alpha$ and $M= \bigcup_{i < \alpha} M_i$. We say that $M\in \K_\lambda$ is a $(\lambda, \alpha)$-limit model if there is $N \in \K_\lambda$ such that $M$ is a $(\lambda, \alpha)$-limit model over $N$. We say that $M\in \K_\lambda$ is a limit model if there is $\alpha < \lambda^+$ limit such that $M$  is a $(\lambda, \alpha)$-limit model.

\end{defin}

\begin{fact}\label{simple}\
\begin{enumerate}
\item If $M \in \K_\lambda$ is universal over $N$ and $M \lea M^* \in \K_\lambda$, then $M^*$ is universal over $N$.  
\item Let $\K$ be an AEC with joint embedding and amalgamation. If $M$ is a limit model of cardinality $\lambda$, then for any $N \in \K_\lambda$ there is $f: N \to M$.
\end{enumerate}
\end{fact}
\begin{proof}
The first assertion is trivial so we prove the second one. 

Fix  $\alpha < \lambda^+$, $\{M_i : i < \alpha\}$ a witness to the fact that $M$ is a $(\lambda, \alpha)$-limit model and let $N \in \K_\lambda$. By the joint embedding property applied to $M_0$ and $N$ and using the L\"{o}wenheim-Skolem-Tarski axiom there is $N^*\in \K_\lambda$ and $g: N \to N^*$ such that $M_0 \lea N^*$. Then since $M_1$ is universal over $M_0$, there is $h: N^* \xrightarrow[M_0]{} M_1$. Hence $f:= h \circ g: N \to M$.  \end{proof}

The following fact gives conditions for the existence of limit models.

\begin{fact}\label{existence}
Let $\K$ be an AEC with joint embedding, amalgamation and no maximal models. If $\K$ is $\lambda$-Galois-stable, then for every $N \in \K_\lambda$ and $\alpha < \lambda^+$ limit there is $M$ a $(\lambda, \alpha)$-limit model over $N$. Conversely, if $\K$ has a limit model of cardinality $\lambda$, then $\K$ is $\lambda$-Galois-stable
\end{fact}
\begin{proof}
The forward direction is claimed in \cite{sh600} and proven in \cite[2.9]{tamenessone}. The backward direction is straightforward.
\end{proof}

As mentioned in the introduction, the uniqueness of limit models of the same cardinality is a very interesting assertion. When the lengths of the cofinalities of the chains are equal, an easy back-and-forth argument gives the following.

\begin{fact}\label{scof} Let $\K$ be an AEC with joint embedding, amalgamation and no maximal models.
If $M$ is a $(\lambda,\alpha)$-limit model and $N$ is a $(\lambda, \beta)$-limit model such that $cf(\alpha)=cf(\beta)$, then $M\cong N$. 
\end{fact}

The question of uniqueness is intriguing when the cofinalities of the lengths of the chains are different. This question has been studied in many papers, among them \cite{shvi}, \cite{van06}, \cite{grvavi}, \cite{extendingframes}, \cite{vand}, \cite{bovan}, \cite{viza} and \cite{vasey18}.

\section{Abelian groups}

In this third section, we study limit models in the class of abelian groups with the subgroup relation. Since this class was studied in great detail in \cite{grp} and \cite{baldwine}, the section will be short and we will cite several times. 

\begin{defin}
Let $\K^{ab}=(K^{ab}, \leq)$ where $K^{ab}$ is the class of abelian groups in the language $L_{ab}=\{0\} \cup \{ +,-\}$ and $\leq$ is the subgroup relation, which is the same as the substructure relation in $L_{ab}$.
\end{defin}

\begin{fact} \label{basic-ab}\
\begin{enumerate}
\item $\K^{ab}$ is an AEC with $\LS(\K^{ab})=\aleph_0$. 
\item $\K^{ab}$ admits intersections. 
\item $\K^{ab}$ has joint embedding, amalgamation and no maximal models.
\item $\K^{ab}$ is a universal class. 
\item $\K^{ab}$ is $(< \aleph_0)$-tame and short.
\end{enumerate}
\end{fact}
\begin{proof}
(1) and (3) are shown in \cite[3.3]{grp} and (2) is clear, so we show the last two assertions:
\begin{enumerate}
 \skipitems{3}
    \item It follows from the fact that $K^{ab}$ is axiomatizable by a set of universal first-order sentences in the language $L_{ab}=\{0\} \cup \{ +,-\}$. It is fundamental that we have ``$-$" in the language.
    \item It follows from (4) and \cite[3.7, 3.8]{vaseyu}.
  \end{enumerate} \end{proof}

The following fact is implied by \cite[3.4, 3.5]{grp}.
\begin{fact}\label{ab-st}
Let $G \leq H$ and $a, b \in H$, the following are equivalent:
\begin{enumerate}
\item There exists $f: cl^H_{\K^{ab}}(G \cup \{a\})\cong_{G} cl^H_{\K^{ab}}(G \cup \{b\})$ such that $f(a)=b$.
\item
 \begin{itemize}
\item $\langle a \rangle \cap  G = 0 = \langle b \rangle \cap  G$, or
\item There are $ n \in \mathbb{N}$ and  $g^* \in G$ such that  $na=g^*=nb$ and $ma, mb \notin G$ for all $m < n$.  
\end{itemize}
\end{enumerate}

In particular, $\K^{ab}$ is $\lambda$-Galois-stable for every $\lambda$ infinite cardinal.
\end{fact}

\begin{remark}
 Since $\K^{ab}$ has joint embedding, amalgamation and no maximal models, $\K^{ab}$ has  limit models in every infinite cardinal by Fact \ref{ab-st} and Fact \ref{existence}.

\end{remark}

Recall that a group $G$ is \emph{divisible} if for each $g \in G$ and $n \in \mathbb{N}$, there is $h\in G$ such that $nh=g$. In the next lemma we show that limit models in $\K^{ab}$ are divisible groups.

\begin{lemma}
If $G$ is a $(\lambda, \alpha)$-limit model, then $G$ is a divisible group.
\end{lemma}
\begin{proof}
Fix $\{G_i : i < \alpha\}$ a witness to the fact that $G$ is a $(\lambda, \alpha)$-limit model. Let $g \in G$ and $n \in \mathbb{N}$, we want to show that $n|g$. Since $G = \bigcup_{i < \alpha} G_i$, there is $i < \alpha$ such that $g\in G_i$. Recall that every group can be embedded as a subgroup into a divisible group (see \cite[\S4.1.4]{fuchs}), so there is $D \in \K_{\lambda}$ divisible group such that $G_i \leq D$. In particular there is $d \in D$ with $nd=g$. Since $G_{i + 1}$ is universal over $G_i$, there is $f: D \xrightarrow[G_i]{} G$. Hence $nf(d)= f(g)=g$ and $f(d) \in G$. 
\end{proof}

Using the following structure theorem for divisible groups we can characterize the limit models of $\K^{ab}$. A proof of this fact appears in \cite[\S 4.3.1]{fuchs}.

\begin{fact}
If $G$ is a divisible group, then we have that:
\[G \cong (\oplus_{\kappa}\mathbb{Q}) \oplus \oplus_{p \text{ prime }} (\oplus_{\kappa_p} \mathbb{Z}(p^\infty))\]
where the cardinal numbers $\kappa$, $\kappa_p$ (for all $p$ prime number) correspond to the ranks $rk_0(G)$, $rk_p(G)$  (for all $p$ prime number)\footnote{The $rk_0(G)$ is the cardinality of a maximal linearly independent subset of elements of infinite order in $G$ and $rk_p(G)$ is the cardinality of a maximal linearly independent subset of elements of order a power of $p$ in $G$. The notion of linear independence in the context of abelian groups differs slightly from that of vector spaces, the reader can consult \cite[p. 91]{fuchs} for the definition of linear independence in this setting.}.
\end{fact}

From it we are able to show our first theorem.

\begin{theorem}\label{abelian}
If $G$ is a $(\lambda, \alpha)$-limit model in $\K^{ab}$, then we have that:
\[  G \cong (\oplus_{\lambda}\mathbb{Q}) \oplus \oplus_{p \text{ prime}} (\oplus_{\lambda} \mathbb{Z}(p^\infty)).\]
\end{theorem}
\begin{proof}
Fix $\{G_i : i < \alpha\}$ a witness to the fact that $G$ is a $(\lambda, \alpha)$-limit model. Observe that $G_0\leq G_0 \oplus (\oplus_{\lambda}\mathbb{Q}) \oplus \oplus_{p \text{ prime}} (\oplus_{\lambda} \mathbb{Z}(p^\infty))$, therefore there is \[f:  G_0 \oplus (\oplus_{\lambda}\mathbb{Q}) \oplus \oplus_{p \text{ prime}} (\oplus_{\lambda} \mathbb{Z}(p^\infty)) \xrightarrow[G_0]{} G.\] In particular, $rk_0(G)= \lambda$ and $rk_p(G)=\lambda$ for all $p$ prime, then by the structure theorem for divisible groups we have that  $G \cong (\oplus_{\lambda}\mathbb{Q}) \oplus \oplus_{p \text{ prime}} (\oplus_{\lambda} \mathbb{Z}(p^\infty))$.
\end{proof}

As a simple corollary we obtain the following.

\begin{cor}
$\K^{ab}$ has uniqueness of limit models for every infinite cardinal.
\end{cor}

\begin{remark}
Fact \ref{ab-st} and  Fact \ref{basic-ab}.(3) together with \cite[3.7, 11.3, 11.7]{vaseyt} imply that $\K^{ab}$ has uniqueness of limit models above $\beth_{(2^{\aleph_0})^+}$, so the result of the above corollary is only new for small cardinals.
\end{remark}

\section{Torsion-free abelian groups}

In this fourth section, we study the class of torsion-free abelian groups with the pure subgroup relation. In the first half of the section we examine basic properties of the class while in the second one we look at limit models. As we will see in this case the theory becomes more interesting. 

\begin{defin}
Let $\K^{tf}=(K^{tf}, \leq_p)$ where $K^{tf}$ is the class of torsion-free abelian groups in the language $L_{ab}=\{0\} \cup \{ +,-\}$ and $\leq_p$ is the pure subgroup relation. Recall that $H$ is a \emph{pure subgroup} of $G$ if for every $n \in \mathbb{N}$ it holds that $nG \cap H  = nH$.
\end{defin}

\subsection{Basic properties} Before analyzing the set of limit models, we obtain a few basic properties for the class of torsion-free abelian groups. As for abelian groups the basic properties of torsion-free abelian groups were studied in \cite{grp} and \cite{baldwine}.
\begin{fact} \label{basictf}\
\begin{enumerate}
\item $\K^{tf}$ is an AEC with $\LS(\K^{tf})=\aleph_0$.
\item$\K^{tf}$ admits intersections. 
\item $\K^{tf}$ has joint embedding, amalgamation and no maximal models.
\end{enumerate}
\end{fact}
 \begin{proof}
(1) and (3) are shown in \cite[3.3]{grp} and \cite{baldwine}  and (2) is known to hold (an argument for this is given in \cite[\S 5.1]{fuchs}).
\end{proof}

The following proposition characterizes the closure operator in $\K^{tf}$, since the proof is a straightforward induction we omit it.
 
\begin{prop}\label{cl}
If $A \subseteq H $, then $cl^{H}_{\K^{tf}}(A)=\bigcup_{n < \omega} A_n$ where:
\begin{itemize}
\item $A_0=A$.
\item $A_{2k+1}=\{-h  : h \in A_{2k} \} \cup \{ \Sigma_{i=0}^{n} h_i : h_0,...,h_n \in A_{2k}, n \in \mathbb{N}\}$. 
\item $A_{2k+2}=\{ h \in H : \text{ there are } h^* \in A_{2k +1} \text{ and } n\in \mathbb{N} \text{ s.t. } nh=h^*\}$.  
\end{itemize}
\end{prop}

Recall the following definition from  \cite[3.1]{vaseyu}.

\begin{defin}
$\K$ is a \emph{pseudo-universal class} if it admits intersections and for any $N_1, N_2\in K$ and $\bar{a}_1 \in N_1$, $\bar{a}_2 \in N_2$, if $\gtp(\bar{a}_1/\emptyset ; N_1)= \gtp(\bar{a}_2/ \emptyset; N_2)$ and $f,g: cl^{N_1}(\bar{a}_1) \cong  cl^{N_2}(\bar{a}_2)$ are such that $f(\bar{a}_1)=g(\bar{a}_1)=\bar{a}_2$, then $f=g$. 
\end{defin}

The reason pseudo-universal classes will be of interest to us is due to the following statement showed in  \cite[3.7]{vaseyu}.

\begin{fact}\label{puni}
If $\K$ is a pseudo-universal class, then $\K$ is $(<\aleph_0)$-tame and short.
\end{fact}

With this let us prove the following lemma.

\begin{lemma}\label{tflocal}
$\K^{tf}$ is a pseudo-universal class. In particular, $\K^{tf}$ is $(<\aleph_0)$-tame and short.
\end{lemma}
\begin{proof}
Let $H \in K^{tf}$, $\bar{a},\bar{b} \in H$ with $\gtp(\bar{a}/\emptyset ; H)= \gtp(\bar{b}/ \emptyset; H)$ and $f,g: cl^H_{\K^{tf}}(\bar{a}) \cong cl^{H}_{\K^{tf}}(\bar{b})$ such that $f(\bar{a})=g(\bar{a})=\bar{b}$. We show by induction that $f\rest_{A_n}= g\rest_{A_n}$ for all $n < \omega$, where the $A_n$'s are obtained by applying Proposition \ref{cl} to $cl^H_{\K^{tf}}(\bar{a})$. The base step is the hypothesis, so we do the induction step. The odd step is straightforward, so we do the even step. Let $h \in  A_{2k+2}$, by definition there is $h^* \in A_{2k + 1}$ and $n\in\mathbb{N}$ such that $nh = h^*$, then since $f, g$ are isomorphisms we have that $nf(h) = f(h^*)$ and $ng(h) = g(h^*)$. By induction hypothesis $f(h^*)=g(h^*)$, so $nf(h)=ng(h)$; using that divisors in torsion-free groups are unique, we obtain that $f(h)=g(h)$. Hence $\K^{tf}$ is pseudo-universal. The fact that $\K^{tf}$ is $(<\aleph_0)$-tame and short follows from Fact \ref{puni}. \end{proof}

In \cite[0.3]{baldwine} the following  key result is obtained.

\begin{fact}\label{st-tf}
$\K^{tf}$ is $\lambda$-Galois-stable if and only if $\lambda^{\aleph_0}=\lambda$. In particular, $\K^{tf}$ is a Galois-stable AEC.
\end{fact}

\subsection{Limit models}
In this subsection we classify the limit models in the class of torsion-free groups. It is clear that they are not divisible groups because if $G$ is not divisible then $G$ can not be a pure subgroup of a divisible group, but as we will show they are the next best thing, at least when the cofinality of the chain is uncountable. The examination of limit models will be done in two cases, we will first look at chains of uncountable cofinality and then at those of countable cofinality.

\begin{remark}
Since $\K^{tf}$ has joint embedding, amalgamation and no maximal models, $\K^{tf}$ has limit models when $\lambda^{\aleph_0}=\lambda$ (and only in those cardinals) by Fact \ref{st-tf} and Fact \ref{existence}.
\end{remark}

Recall the following characterization of algebraically compact groups \cite[\S 6.1.3]{fuchs}. For more on algebraically compact groups the reader can consult \cite[\S 6]{fuchs}.

\begin{defin}
A group $G$ is algebraically compact if given $\mathbb{E}=\{f_i(x_{i_0},...,x_{i_{n_i}})=a_i : i < \omega \}$ a set of linear equations over $G$, $\mathbb{E}$ is finitely solvable in $G$ if and only if  $\mathbb{E}$ is solvable in $G$.
\end{defin}

%\begin{fact}
%The following are equivalent:
%\begin{enumerate}
%\item $G$ is algebraically compact.
%\item $G$ is a summand in every group that contains it as a pure subgroup.
%\item $G$ is pure-injective.

%\end{enumerate}
%\end{fact}

\begin{lemma}\label{limit-ac}
If $G$ is a $(\lambda,\alpha)$-limit model and $cf(\alpha)\geq \omega_1$, then $G$ is algebraically compact.
\end{lemma}
\begin{proof} Fix $\{G_\beta : \beta < \alpha\}$ a witness to the fact that $G$ is a $(\lambda, \alpha)$-limit model.
 Let $\mathbb{E}=\{f_i(x_{i_0},...,x_{i_{n_i}})=a_i : i < \omega \}$ a set of linear equations finitely solvable in $G$. Since $cf(\alpha)\geq \omega_1$ there is $\beta^*  < \alpha$ such that $\{a_i : i < \omega \} \subseteq G_{\beta^*}$. Add new constants $\{ c_i : i < \omega \}$ and consider:
\[\Sigma = \{ f_i(c_{i_0},...,c_{i_{n_i}})=a_i : i < \omega \} \cup ED(G_{\beta^*}) \cup T_{tf} \cup \{ \neg \exists x( nx=g): G_{\beta^*} \vDash \neg \exists x (nx =g), n\in \mathbb{N}, g\in G_{\beta^*} \},\]
where $T_{tf}$ is the first-order theory of torsion-free abelian groups and $ED(G_{\beta^*}) $ is the elementary diagram of $G_{\beta^*}$. 

Since $\mathbb{E}$ is finitely solvable in $G$ and $G_{\beta^*} \leq_p G$, it is easy to show that any finite subset of $\Sigma$  is realized in $G$. Then by compactness and L\"{o}wenheim-Skolem-Tarski there is $H \in K^{tf}_\lambda$ such that $G_{\beta^*}\leq_p H$ ($G_{\beta^*}$ is a pure subgroup by the last element in the definition of $\Sigma$) and $H \vDash \{ f_i(c_{i_0},...,c_{i_{n_i}})=a_i : i < \omega\}$ . Using the fact that $G_{\beta^* + 1}$ is universal over $G_{\beta^*}$, there is $f: H \xrightarrow[G_{\beta^*}]{} G_{\beta^* +1}$ and it is easy to show that $\{ f(c_i^{H}) : i < \omega\}$ is a set of solutions to $\mathbb{E}$ which is contained in $G$.  \end{proof}

As a simple corollary we obtain a new proof for the following well-known assertion, the assertion without the torsion-free hypothesis appears for example in  \cite[\S 6 1.10]{fuchs}.
\begin{cor}
Every torsion-free group can be embedded as a pure subgroup in  a torsion-free algebraically compact group.
\end{cor}
\begin{proof}
Follows from the joint embedding property, Fact \ref{simple} and the previous lemma.
\end{proof}

Before proving a theorem parallel to Theorem \ref{abelian}, we prove the following proposition. In it the group $\mathbb{Z}_{(p)}$ will play a crucial role, recall that $\mathbb{Z}_{(p)}=\{ n/m : (m,p)= 1 \}$. 

\begin{prop}\label{zpdim}
If $G$ is a $(\lambda, \alpha)$-limit model, then $dim_{\mathbb{F}_p}(G/pG)=\lambda$ for all $p$ prime.\footnote{Notice that the proposition includes the case when the cofinality of $\alpha$ is countable.}
\end{prop}
\begin{proof} Fix $\{G_i : i < \alpha\}$ a witness to the fact that $G$ is a $(\lambda, \alpha)$-limit model. Notice that $G_0 \leq_{p} G_0 \oplus (\oplus_{\lambda} \mathbb{Z}_{(p)})$, then using that $G_1$ is universal over $G_0$, there is $f:  G_0 \oplus (\oplus_{\lambda} \mathbb{Z}_{(p)})  \xrightarrow[G_0]{} G$. In particular, we may assume that $(\oplus_{\lambda} \mathbb{Z}_{(p)}) \leq_{p} G$. 

 \underline{Claim}: $\{ e_i : i < \lambda \} \subseteq (\oplus_{\lambda} \mathbb{Z}_{(p)}) \subseteq G$ satisfy that for every $g \in G$, $A \subseteq_{fin} \lambda$ and  $(n_i)_{i\in A} \in \{0,...,p-1\}^{|A|} \backslash \{\bar{0}\}$ the following holds:
\[ \Sigma_{i\in A} n_ie_i\neq pg .\]

Where each $e_i$ is the $i^{th}$-element of the canonical basis.

 \underline{Proof of Claim}: Suppose for the sake a contradiction that it is not the case, then there is  $g \in G$, $A \subseteq_{fin} \lambda$ and  $(n_i)_{i \in A} \in \{0,...,p-1\}^{|A|}\backslash \{\bar{0}\}$ such that \[ \Sigma_{i\in A} n_ie_i=pg .\]

Since  $(\oplus_{\lambda} \mathbb{Z}_{(p)}) \leq_{p} G$ and $G\in K^{tf}$, we have that $g \in (\oplus_{\lambda} \mathbb{Z}_{(p)})$. Then $g=\Sigma_{i\in B} g_i$ for $B \subseteq_{fin} \lambda$ and unique $(g_i)_{i\in B} \in \mathbb{Z}_{(p)}^{|B|}$. Hence using the above equality it follows that $n_i=pg_i$ for each $i \in A$. Then $p$ would divide the denominator of $g_i$ for some $i \in A$, contradicting the fact that each $g_i \in \mathbb{Z}_{(p)}$, or $g=0$, contradicting the linear independence of the $e_i$'s.$\dagger_{\text{Claim}}$

From the above claim it follows that $\{ e_i + pG : i < \lambda\}$ is a linearly independent set over $\mathbb{F}_p$. Hence  $dim_{\mathbb{F}_p}(G/pG)=\lambda$. \end{proof}

The following fact puts together the information from \cite[\S 1]{ef} that we will need in this paper.\footnote{ We recommend the reader to take a look at \cite[\S 1]{ef} or \cite[\S 6.3]{fuchs}.}

\begin{fact}\label{ac}
If $G$ is a torsion-free algebraically compact group, then:
\[ G \cong (\oplus_{\delta} \mathbb{Q}) \oplus \Pi_{p \text{ prime}} \overline{(\oplus_{\beta_p} \mathbb{Z}_{(p)})}.\]
Where:
\begin{enumerate}
 \item $\beta_p =dim_{\mathbb{F}_p}(G/pG)$ for all $p$ prime  (\cite[1.7.a]{ef}).
\item $\delta = rk_0(G_d)$, where $G_d$ is the maximal divisible subgroup of $G$ (\cite[1.10]{ef}).
\item $\mathbb{Z}_{(p)}=\{ n/m : (m,p)= 1 \}$ for $p$ prime and the overline refers to the completion\footnote{For the reader familiar with abelian group theory, this is precisely the pure-injective hull (see \cite[\S 6.4]{fuchs}).} (look at the discussion between \cite[1.4]{ef} and \cite[1.6]{ef}).
\end{enumerate}
\end{fact}

\begin{lemma}\label{pureac}
If $G$ is a $(\lambda, \alpha)$-limit model and $G$ is algebraically compact, then \[G \cong   (\oplus_{\lambda} \mathbb{Q}) \oplus \Pi_{p \text{ prime}} \overline{(\oplus_{\lambda} \mathbb{Z}_{(p)})}. \]
\end{lemma}
\begin{proof} Fix $\{G_i : i < \alpha\}$ a witness to the fact that $G$ is a $(\lambda, \alpha)$-limit model. 
Since by hypothesis $G$ is algebraically compact, by Fact \ref{ac} it is enough to show that  $\beta_p=\lambda$ for all $p$ prime and that $\delta=\lambda$.

By Fact \ref{ac}.(1) and Proposition \ref{zpdim} it follows that $\beta_p= dim_{\mathbb{F}_p}(G/pG)=\lambda$ for all $p$ prime, so we just need to show that $\delta=\lambda$. Observe that $G_0 \leq_p G_0 \oplus   (\oplus_{\lambda} \mathbb{Q})$, then there is $f:  G_0 \oplus   (\oplus_{\lambda} \mathbb{Q}) \xrightarrow[G_0]{} G$, from which it follows that $ rk_0(G_d)= \lambda$ since $f[   (\oplus_{\lambda} \mathbb{Q})] \subseteq G_d$. Hence by Fact \ref{ac}.(2), we have that $\delta=\lambda$.  
\end{proof}

With this we obtain our main result on limit models of uncountable cofinality.

\begin{theorem}\label{purec}
If $G$ is a $(\lambda, \alpha)$-limit model and $cf(\alpha) \geq \omega_1$, then \[G \cong   (\oplus_{\lambda} \mathbb{Q}) \oplus \Pi_{p \text{ prime}} \overline{(\oplus_{\lambda} \mathbb{Z}_{(p)})}. \]
\end{theorem}
\begin{proof}
By Lemma \ref{limit-ac} $G$ is algebraically compact. Then the result follows from Lemma \ref{pureac}.
\end{proof}

The following corollary follows directly from Theorem \ref{purec}.

\begin{cor}
If $G$ is a $(\lambda, \alpha)$-limit model and $H$ is a $(\lambda, \beta)$-limit model such that $cf(\alpha), cf(\beta) \geq \omega_1$, then $G\cong H$. 
\end{cor}

\begin{remark}
Since $\K^{tf}$ has joint embedding, amalgamation, no maximal models and is $(<\aleph_0)$-tame, by \cite[3.7]{vaseyt} non-splitting has weak continuity and then by \cite[11.3, 11.7]{vaseyt} it follows that  $\K^{tf}$ has uniqueness of limit models for large $\lambda$ and $cf(\alpha)$. Therefore, the result of the above corollary is only new for small cardinals.
\end{remark}

The next corollary follows from the above corollary doing a similar construction to \cite[2.8.(3)]{grva}.

\begin{cor}\label{sat}
If $G$ is a $(\lambda, \alpha)$-limit model and $cf(\alpha)\geq \omega_1$, then $G$ is $\lambda$-Galois-saturated.\footnote{Recall that $G$ is $\lambda$-Galois-saturated if for every $H \lea G$ and $p \in \gS(H)$ such that $\|H\| < \lambda$, $p$ is realized in $G$. $G$ is Galois-saturated if it is $\| G\|$-Galois-saturated.} 
\end{cor}

This finishes the characterization of $G$ when $G$ is a $(\lambda, \alpha)$-limit model and the cofinality of $\alpha$ is uncountable, we know tackle the question when the cofinality of $\alpha$ is countable. Regarding it, we will only have negative results, i.e., we will show that if $G$ is a $(\lambda, \alpha)$-limit model then $G$ is not algebraically compact. In order to do that, we will use some deep results on AECs which appear in \cite{grva} and \cite{vas16a}. Realize that since limit models with lengths of chains of the same cofinality are isomorphic, we only need to study $(\lambda, \omega)$-limit models.

The proof will be divided into two parts. In the first we will use \cite{grva} and \cite{vas16a} to show that for $\lambda$ big $(\lambda, \omega)$-limit models are not algebraically compact and in the second we will reflect the big groups into smaller cardinalities.

The following fact contains the information we will need from \cite{grva} and \cite{vas16a}. For the readers not familiar with the theory of AECs this can be taken as a black box.

\begin{fact}\label{saec}
Assume that $\K$ has joint embedding, amalgamation, no maximal models, $\LS(\K)=\aleph_0$ and is $(<\aleph_0)$-tame.
Let $\lambda \geq \beth_{(2^{\aleph_0})^+ +  \omega}$ be such that $\K$ is $\lambda$-Galois-stable and there is a Galois-saturated model of cardinality $\lambda$. If every limit model of cardinality $\lambda$ is Galois-saturated, then $\K$ is $\chi$-Galois-stable for every $\chi\geq \lambda$.
\end{fact}
\begin{proof}[Proof sketch]
By \cite[3.2]{grva} $\K$ does not have the $\aleph_0$-order property of length $\beth_{(2^{\aleph_0})^+}$. Then by \cite[3.18]{grva} $\K$ has no long splitting chains in $\lambda$. Since $\K$  has no long splitting chains in $\lambda$, is $\lambda$-Galois-stable and is $(<\aleph_0)$-tame by \cite[5.6]{vas16a} we can conclude that $\K$ is $\chi$-Galois-stable for every $\chi\geq \lambda$.
\end{proof}

\begin{lemma}\label{large} Let  $\lambda \geq \beth_{(2^{\aleph_0})^+ +  \omega}$. If $G$ is a $(\lambda, \omega)$-limit model, then $G$ is not algebraically compact.
\end{lemma}
\begin{proof}
Since $G$ is a $(\lambda, \omega)$-limit model, it follows that $\K^{tf}$ is $\lambda$-Galois-stable by Fact \ref{existence}.

Assume for the sake of contradiction that $G$ is algebraically compact, then by Lemma \ref{pureac} $G \cong   (\oplus_{\lambda} \mathbb{Q}) \oplus \Pi_{p \text{ prime}} \overline{(\oplus_{\lambda} \mathbb{Z}_{(p)})}$. Then by Theorem \ref{purec} $\K^{tf}$ has uniqueness of limit models of cardinality $\lambda$. Hence every limit model of cardinality $\lambda$ is Galois-saturated by \cite[2.8.(3)]{grva}.

By Fact \ref{basictf} and Lemma \ref{tflocal} $\K^{tf}$ has  joint embedding, amalgamation, no maximal models, $\LS(\K^{tf})=\aleph_0$ and is $(<\aleph_0)$-tame. Then by Fact \ref{saec} $\K^{tf}$ is $\chi$-Galois-stable for every $\chi \geq \lambda$. But this contradicts Fact \ref{st-tf}, since there is $\chi \geq \lambda$ such that $\chi^{\aleph_0}\neq \chi$.
\end{proof}

\begin{lemma}\label{small} Let  $\lambda <  \beth_{(2^{\aleph_0})^+ +  \omega}$. If $G$ is a $(\lambda, \omega)$-limit model, then $G$ is not algebraically compact.
\end{lemma}
\begin{proof}
Since $G$ is a $(\lambda, \omega)$-limit model, it follows that $\K^{tf}$ is $\lambda$-Galois-stable by Fact \ref{existence}.

Let $\mu \geq \beth_{(2^{\aleph_0})^+ +  \omega}$ such that $\mu^{\aleph_0}=\mu$, by Fact \ref{st-tf} $\K^{tf}$ is $\mu$-Galois-stable. Let $G^*$ a $(\mu, \omega)$-limit model witnessed by $\{ G^*_i : i < \omega \}$. By Lemma \ref{large} $G^*$ is not algebraically compact, so there is $\mathbb{E}=\{f_k(x_{k_0},...,x_{k_{n_k}})=a_k : k < \omega \}$ a set of linear equations finitely solvable in $G^*$ but not solvable in $G^*$.

We build $\{r_i : i < \omega\} \subseteq \mathbb{N}$, $\{S_i : i < \omega \}$ and $\{ H_i : i < \omega \}$ by induction such that:
\begin{enumerate}
\item $\{r_i : i < \omega \} $ is strictly increasing.
\item $a_i \in H_i$.
\item $S_i \subseteq H_i$ and $S_i$ is a finite set.
\item $S_i$ has a solution to $\{f_k(x_{k_0},...,x_{k_{n_k}})=a_k : k \leq i \}$. 
\item $H_i \leq_p G_{r_i}^*$.
\item $H_i \in \K_\lambda^{tf}$.
\item $H_{i+1}$ is universal over $H_i$.
\end{enumerate} 

Before we do the construction, let us show that this is enough. Let $H_\omega := \bigcup_{i < \omega} H_i$, by (6) and (7) it follows that $H_\omega$ is a $(\lambda, \omega)$-limit model. Since limit models of the same cofinality are isomorphic by Fact \ref{scof}, it follows that $H_\omega \cong G$. So it is enough to show that $H_\omega$ is not algebraically compact. Assume for the sake of contradiction that $H_\omega$ is algebraically compact. Since $\mathbb{E}=\{f_k(x_{k_1},...,x_{k_{n_k}})=a_k : k < \omega \}$ is finitely solvable in $H_\omega$ by (4), it follows that there is $\ba \in H_\omega^\omega$ a solution for $\mathbb{E}$. But this contradicts the fact that $\mathbb{E}$ is not solvable in $G^*$, since $H_\omega \leq_p G^*$ by (5). Therefore, $H_\omega$ is not algebraically compact.

Now let us do the construction.

\fbox{Base} 	Let $\{b_0,...,b_l \} \subseteq G^*$ a solution to $f_0(x_{0_0},...,x_{0_{n_0}})=a_0$, this exists by finite solvability of $\mathbb{E}$ in $G^*$, and $r < \omega$ such that $\{b_0,...,b_l, a_0 \} \subseteq G^*_r$.  Let $r_0:= r$, $S_0 := \{b_0,...,b_l \}$ and applying L\"{o}wenheim-Skolem-Tarski axiom to  $\{b_0,...,b_l, a_0 \}$ in $G_{r_0}^*$ we get $H_0 \in \K_\lambda^{tf}$ such that $H_0 \leq_p G_{r_0}^*$ and $\{b_0,...,b_l, a_0 \}  \subseteq H_0$. It is easy to see that this works.

\fbox{Induction  step} By construction there are $r_i \in\mathbb{N}$ and $H_i \leq_p G_{r_i}^*$. Since $\K^{tf}$ is $\lambda$-Galois-stable we can build $H \in \K_\lambda^{tf}$ such that $H$ is universal over $H_i$ by Fact \ref{existence}. Using that $H_i \leq_p  G_{r_i}^*$, the amalgamation property and that $G_{r_i + 1}^*$ universal over $G_{r_i}^*$, there is $f: H \xrightarrow[H_i]{} G_{r_i +1}^*$.

Let $\{b_0,...,b_l\} \subseteq G^*$ a solution to $\{f_k(x_{k_0},...,x_{k_{n_k}})=a_k : k \leq i+1 \}$ and take $r \geq r_i  + 1$ such that $\{b_0,...,b_l,a_{i+1}\} \subseteq G^*_{r}$. Let $r_{i+1}:=r$, $S_{i+1}:= \{b_0,...,b_l \}$ and applying L\"{o}wenheim-Skolem-Tarski axiom to $f[H] \cup  \{b_0,...,b_l,a_{i+1}\} $ in $G_{r_{i+1}}^*$ we get $H_{i+1} \in \K_\lambda^{tf}$ such that $H_{i+1} \leq_p G_{r_{i+1}}^*$ and $f[H] \cup  \{b_0,...,b_l,a_{i+1}\} \subseteq H_{i+1}$.  Using that $H_i\leq_p f[H] \leq_{p} H_{i+1}$ and that $f[H]$ is universal over $H_i$, it is easy to show that (1) through (7) hold.
\end{proof}

Putting together the last two lemmas we obtain the following.

\begin{theorem}\label{limit-count}
If $G$ is a $(\lambda,\omega)$-limit model, then $G$ is not algebraically compact.
\end{theorem}
\begin{proof}
If $\lambda \geq  \beth_{(2^{\aleph_0})^+ +  \omega}$ it follows from Lemma \ref{large} and if $\lambda <  \beth_{(2^{\aleph_0})^+ +  \omega}$ it follows from Lemma \ref{small}.
\end{proof}

\begin{remark}
After discussing Theorem \ref{limit-count} with Sebastien Vasey, he realized that by applying \cite[4.12]{vaseyt} instead of \cite[3.18]{grva} one could prove Theorem \ref{limit-count} without dividing the proof into cases. The proof using \cite[4.12]{vaseyt} is similar to that of Lemma \ref{large}. We decided to keep our original argument since the proof presented here shows how to transfer the failure of being algebraically compact and since we believe that showing that there are cofinally many $(\lambda, \omega)$-limit models that are not algebraically compact is provable using only group theoretic methods.
\end{remark}

Since $(\lambda, \omega)$-limit models are not algebraically compact we ask:

\begin{question}
Is there a natural class of groups that contain the $(\lambda, \omega)$-limit models?
\end{question}

Regarding the structure of $(\lambda, \omega)$-limit models, using the fact that every group is a direct sum of a divisible group and a reduced group\footnote{Recall that a group $H$ is reduced if its only divisible subgroup is $0$.} (see \cite[\S 4.2.5]{fuchs}), it is straightforward to show that if $G$ is a $(\lambda, \omega)$-limit model, then $G \cong   (\oplus_{\lambda} \mathbb{Q}) \oplus G_r$ where $G_r \cong G/G_d$, $G_d$ is the maximal divisible subgroup of $G$ and $G_r$ is reduced. So it is natural to ask the following.

\begin{question}
Is there a structure theorem for $(\lambda, \omega)$-limit models similar to that of Theorem \ref{purec}?
\end{question}

Let us conclude with the main theorem of this section.

\begin{theorem} \label{tf}
If $G$ is a $(\lambda, \alpha)$-limit model in $\K^{tf}$, then we have that:
\begin{enumerate}
\item If the cofinality of $\alpha$ is uncountable, then $G \cong   (\oplus_{\lambda} \mathbb{Q}) \oplus \Pi_{p \text{ prime}} \overline{(\oplus_{\lambda} \mathbb{Z}_{(p)})}$.
\item If the cofinality of $\alpha$ is countable, then $G$ is not algebraically compact.
\end{enumerate}
In particular, $\K^{tf}$ does not have uniqueness of limit models for any infinite cardinal.
\end{theorem}
\begin{proof}
The first part is Theorem \ref{purec} and the second one is Theorem \ref{limit-count}. The ``in particular" follows from the fact that limit models with chains of uncountable cofinality are algebraically compact by (1), while those with chains of countable cofinality are not algebraically compact by (2).
\end{proof}

\section{Finitely Butler Groups}

In this last section, we look at some basic properties of the  class of finitely Butler groups. The results in this section are weaker than those of the previous two sections and in some sense incomplete, but we decided to present them since we see this section as a stepping stone and moreover finitely Butler groups had never been isolated as an AEC.

Butler groups were introduced by Butler in \cite{butler}, while finitely Butler groups were first studied in \cite{bican} and given a name in \cite{fuchsv}. We follow the exposition of \cite[\S 14]{fuchs} and recommend the reader to consult it for further details.  

\begin{defin}
A torsion-free group $G$ of finite rank\footnote{Given $G$ a torsion-free group, the rank of $G$ is $rk_0(G)$ (see footnote 6 for the definition).}  is a \emph{Butler group} if $G$ is a pure subgroup of a finite rank completely decomposable group. Recall that a  torsion-free group is \emph{completely decomposable} if and only if it is the direct sum of groups of rank one. 
\end{defin}

\begin{defin}
A torsion-free group $G$ is a \emph{finitely Butler group} ($B_0$-group) if every pure subgroup of finite rank of $G$ is a Butler group.
\end{defin}

Let us introduce the class we will study. 

\begin{defin}
Let $\K^{B_0}=(K^{B_0}, \leq_p)$ where $K^{B_0}$ is the class of finitely Butler groups in the language $L_{ab}=\{0\} \cup \{ +,-\}$ and $\leq_p$ is the pure subgroup relation.
\end{defin}

\begin{remark}\label{easy}
Notice that if $G \in K^{B_0}$ and $H \leq_p G$, then $H \in K^{B_0}$.
\end{remark}

Our first assertion is that indeed $\K^{B_0}$ is an AEC.

\begin{lemma}
 $\K^{B_0}=(K^{B_0}, \leq_p)$ is an AEC with $\LS(\K^{B_0})=\aleph_0$ that admits intersections.
\end{lemma}
\begin{proof}

From the closure under pure subgroups and the fact that $\K^{tf}$ is an AEC, it follows that $\K^{B_0}$ satisfies all the axioms of an AEC except the first Tarski-Vaught axiom. We show that it holds.\footnote{This is exercise \cite[\S 14.4.1]{fuchs}.}

Let $\{ G_i : i < \delta \}$ such that $G_i \leq_p G_j$ for all $i< j$ and $G =\bigcup_{i < \delta} G_i$. It is clear that $G_i \leq_p G$ for all $i< j$, so we only need to show that $G \in K^{B_0}$, so let $H \leq_p G$ of finite rank.

 Take $X$ a finite maximal linearly independent subset of $H$, it exists because $H$ has finite rank. Since $X$ is finite, there is $i < \delta$ such that $X \subseteq G_i$. Since $X$ is maximal linearly independent $H \subseteq span_{\mathbb{Q}}(X)$. Then using that $G_i \leq_p G$ and $G_i$ is torsion-free, it follows that $H\leq_p G_i$. Therefore, since $G_i \in K^{B_0}$, we conclude that $H$ is a Butler group.

Moreover, the class admits intersections because $\K^{tf}$ admits intersections and the closure of $\K^{B_0}$ under pure subgroups.
\end{proof}

\begin{fact}
$\K^{B_0}$ has joint embedding and no maximal models.
\end{fact}
\begin{proof}
By \cite[\S 14.5.(B)]{fuchs} $\K^{B_0}$ is closed under direct sums so the result follows.
\end{proof}

Regarding the amalgamation property, we are only able to provide the following partial solution. We actually think that the amalgamation property might not hold for the class.

\begin{lemma}
If $G \in K^{B_0}$ and $G$ is divisible, then $G$ is an amalgamation base, i.e., if $G\leq_p H_i \in K^{B_0}$ for $i\in \{1,2\}$, then there are $H \in K^{B_0}$ and $f_i: H_i \to H$ for $i\in \{1,2\}$ such that $f_1\rest_G =f_2\rest_G$.
\end{lemma}
\begin{proof}
Let $G \leq_p H_i$ for $i\in \{1,2\}$. Let $H := H_1 \oplus H_2/ G^*$ where $G^*:=\{(g,-g) : g \in G\}$, $f_1: H_1\to H$ be $f(h)=(h,0) + G^*$ and $f_2: H_2 \to H$ be $f(h)=(0, h) +  G^*$. In \cite[3.27]{grp} it is shown that $H \in K^{tf}$, $f_1,f_2$ are pure embeddings and $f_1\rest_G =f_2\rest_G$. So we only need to show that $H \in K^{B_0}$.

Let $E \subseteq H_1 \oplus H_2$ such that $E/G^* \leq_p  H_1 \oplus H_2/ G^*$ and $E/G^*$ has rank $n$. Take $\{e_i + G^* : i < n\}$ a maximal linearly independent subset of  $E/G^*$.

Observe that $E \leq_p H_1 \oplus H_2$, because $G^* \leq_p H_1 \oplus H_2$ and $E/G^* \leq_p  H_1 \oplus H_2/ G^*$. Moreover, $cl^{E}_{\K^{B_0}}(\{e_0,...,e_{n-1}\}) \leq_p H_1 \oplus H_2$, $cl^{E}_{\K^{B_0}}(\{e_0,...,e_{n-1}\})$ has finite rank and $H_1 \oplus H_2 \in K^{B_0}$ (see \cite[\S 14.5.(B)]{fuchs}), so it follows that  $cl^{E}_{\K^{B_0}}(\{e_0,...,e_{n-1}\})$  is a Butler group (where the closure is the one described in Proposition \ref{cl} by Remark \ref{easy}).

 \underline{Claim}: $E=  G^* + cl^{E}_{\K^{B_0}}(\{e_0,...,e_{n-1}\})$.

 \underline{Proof of Claim}: Let $e \in E$, since  $\{e_i + G^* : i < n\}$  is maximal linearly independent $e + G^* \in span_{\mathbb{Q}}(\{e_i + G^* : i < n\})$, then there are $\{ m,k_0,...,k_{n-1}\} \subseteq \mathbb{N}$ and $g^*_0 \in G^*$ such that: \[me= \Sigma_{i=0}^{n-1} k_ie_i + g^*_0.\]

Since $G$ is divisible, $G^*$ is divisible so there is $g^*_1 \in G^*$ such that $mg^*_1= g^*_0$. Then $m(e-g^*_1)=\Sigma_{i=0}^{n-1} k_ie_i$, thus $e-g^*_1 \in  cl^{E}_{\K^{B_0}}(\{e_0,...,e_{n-1}\})$. Hence $e \in  G^* + cl^{E_{\K^{B_0}}}(\{e_0,...,e_{n-1}\})$.$\dagger_{\text{Claim}}$

Then $E/G^* \cong G^* + cl^{E}_{\K^{B_0}}(\{e_0,...,e_{n-1}\})/ G^* \cong  cl^{E}_{\K^{B_0}}(\{e_0,...,e_{n-1}\})/  cl^{E}_{\K^{B_0}}(\{e_0,...,e_{n-1}\}) \cap G^*$. By the fact that  torsion-free epimorphic images of Butler groups are Butler groups (see \cite[\S 14.1.6]{fuchs}) and that $cl^{E}_{\K^{B_0}}(\{e_0,...,e_{n-1}\})$ is a Butler group, we conclude that $E/G^*$ is a Butler group. Hence $H \in K^{B_0}$. \end{proof}

The next proposition is straightforward, but we include it because of its strong consequences.
\begin{prop}\label{tp-b0}
If $G, H \in K^{B_0}$, $\ba \in G$, $\bb \in H$ and $A \subseteq G,H$, then $\gtp_{\K^{B_0}}(\ba/A; G)= \gtp_{\K^{B_0}}(\bb/A; H)$ if and only if $\gtp_{\K^{tf}}(\ba/A; G)= \gtp_{\K^{tf}}(\bb/A; H)$.

\end{prop}
\begin{proof}
Since $\K^{B_0}$ is closed under pure subgroups by Remark \ref{easy}, using the minimality of the closures, it is easy to show that for all $H'\in K^{B_0}$ and $B \subseteq H'$ it holds that $cl_{\K^{B_0}}^{H'}(B)=cl_{\K^{tf}}^{H'}(B)$. Then using that $\K^{B_0}$ and $\K^{tf}$ admit intersections and Fact \ref{tp-int} the result follows.
\end{proof}

\begin{cor}\label{st-fb}\
\begin{itemize}
\item $\K^{B_0}$ is $(<\aleph_0)$-tame and short.
\item If $\lambda=\lambda^{\aleph_0}$, then $\K^{B_0}$ is $\lambda$-Galois-stable. In particular, $\K^{B_0}$ is a Galois-stable AEC.
\end{itemize}
\end{cor}
\begin{proof}
The proof follows directly from Proposition \ref{tp-b0} and the fact that $\K^{tf}$ satisfies both of the properties we are trying to show.
\end{proof}

\begin{question}
Do we have as in $\K^{tf}$ that: if $\K^{B_0}$ is $\lambda$-Galois-stable, then $\lambda=\lambda^{\aleph_0}$?
\end{question}

We were unable to answer the above question, but we have a partial solution (see Lemma \ref{nst}). In order to present it, we will need some results from \cite[\S 12.1]{fuchs} and the following definitions.

\begin{defin} Let $G$ be a torsion-free abelian group and $a \in G$:
\begin{itemize}
\item Given a prime $p$ the \emph{$p$-height} of $a$ (denoted by $h_p(a)$) is the maximum $n\in \mathbb{N}$ such that $p^n|a$ or $\infty$ if the maximum does not exist. 
\item The \emph{characteristic} of $a$ is $\chi_G(a)=(h_{p_n}(a))_{n < \omega}$ where $\{p_n : n< \omega\}$ is an increasing enumeration of the prime numbers.
\item Given $\eta, \nu \in (\mathbb{N}\cup\{\infty\})^{\omega}$ we define the equivalence relation $\sim$ as $\eta \sim \nu$ if and only if $\eta$ and $\nu$ differ on finitely many natural numbers and when they differ they are both finite. A \emph{type $\textbf{t}$} is an element of $(\mathbb{N}\cup\{\infty\})^{\omega}/\sim$ and the \emph{type of $a$} is $\textbf{t}_{G}(a)=\chi_G(a)/\sim$.
\item We say that $G$ has type $\textbf{t}$, if for every $b \neq 0\in G$ it holds that $\textbf{t}=\textbf{t}_{G}(b)$.
\end{itemize}

\end{defin}

The proof of the following lemma uses similar ideas to those of \cite[3.7]{kojsh}.

\begin{lemma}\label{nst}
If $\lambda< 2^{\aleph_0}$, then $\K^{B_0}$ is not $\lambda$-Galois-stable.
\end{lemma}
\begin{proof}
Let $G\in \K_\lambda^{B_0}$ and $\{ \textbf{t}_\eta : \eta \in 2^\omega  \}$ an enumeration of all the types (in the sense of the previous definition). For each $\eta \in 2^{\omega}$, let $G_\eta$ a group of rank one with type $\textbf{t}_\eta$, it exists by \cite[\S 12.1.1]{fuchs}. Let $H= G \oplus (\oplus_{\eta \in 2^{\omega}} G_\eta)$. Since $\K^{B_0}$ is closed under direct sums  (see \cite[\S 14.5.(B)]{fuchs}) and rank one groups are in $\K^{B_0}$, because they are completely decomposable, we have that $H \in K^{B_0}$. 

For each $\eta \in 2^{\omega}$ take $a_\eta \in G_{\eta}$ with $a_\eta \neq 0$ and let $p_\eta:= \gtp(a_\eta/G; H)$. We show that all the Galois-types in the set $\{p_\eta : \eta \in 2^\omega \}$ are different.

\underline{Claim:} If $\eta \neq \nu \in 2^\omega$, then $p_\eta \neq p_\nu$.\\
 \underline{Proof of Claim}: Suppose for the sake of contradiction that $\gtp(a_\eta/G; H) = \gtp(a_\nu/G; H)$, then by Fact \ref{tp-int} there is $f: cl^H_{\K^{B_0}}(\{a_\eta\}\cup G) \cong_{G} cl^H_{\K^{B_0}}(\{a_\nu\} \cup G)$ with $f(a_\eta) =a_\nu$. Then since the closures give rise to pure subgroups of $H$ we have that $\chi_{H}(a_\eta)=\chi_H(a_\nu)$, so $\textbf{t}_{H}(a_\eta) =\textbf{t}_H(a_\nu)$. This contradicts the fact that $\textbf{t}_{H}(a_\eta) = \textbf{t}_{G_\eta}(a_\eta)=\textbf{t}_\eta  \neq \textbf{t}_\nu=\textbf{t}_{G_\nu}(a_\nu)= \textbf{t}_H(a_\nu)$, the first and last equality follow from the fact that $G_\eta, G_\nu \leq_p H$. $\dagger_{\text{Claim}}$

Therefore, $|\gS(G)|\geq 2^{\aleph_0}$. Since $\lambda < 2^{\aleph_0}$, $\K$ is not $\lambda$-Galois-stable.
\end{proof}

As we mentioned in the introduction we are interested in limit models, therefore we ask the following:

\begin{question}
Do limit models exist  in $\K^{B_0}$? If they exist, what is their structure?
\end{question}

Regarding the first part of the question, realize that if $\K^{B_0}$ has the amalgamation property, then by Corollary \ref{st-fb} and  Fact \ref{existence} limit models would exist. As for the second part, even if they existed the techniques to characterize them would have to be different from the ones presented in section four since finitely Butler groups do not seem to be first-order axiomatizable.

Besides the function of this article as a pool of examples of limit models in the context of AECs. We believe that the study of limit models (in different classes of groups) as a classes of infinite rank groups could be an interesting area of research on its own. We think this is possible since limit model are tame enough to be analyzable, but their theory is nontrivial as showcased in this article. A good place to look for new classes of limit models is \cite{baldwine}.

%\bibliography{}{}
%\bibliographystyle{amsalpha}

\end{document}